\documentclass[11pt,oneside,reqno]{amsart}
\usepackage{graphicx}
\numberwithin{equation}{section}

\newcommand{\bea}{\begin{eqnarray}}
\newcommand{\eea}{\end{eqnarray}}
\newcommand{\ba}{\begin{array}}
\newcommand{\ea}{\end{array}}
\newcommand{\edc}{\end{document}}
\newcommand{\bc}{\begin{center}}
\newcommand{\ec}{\end{center}}
\newcommand{\be}{\begin{equation}}
\newcommand{\ee}{\end{equation}}

\def\bc{{\mathbb C}}

\def\s{\sigma}

\def\h{{\mathbf{h}}}

\newtheorem{thm}{Theorem}[section]

\newtheorem{prop}[thm]{Proposition}

\theoremstyle{remark}
\newtheorem{rem}{Remark}[section]

\date{\today}
\begin{document}
\title[Gibbs Measures on an arbitrary order Cayley tree]
{Gibbs Measures with memory of length 2 on an arbitrary order Cayley tree}
\author{Hasan Ak\i n}
\address{Hasan Ak\i n\\ Ceyhun Atuf Kansu Cad. Cankaya, Ankara, Turkey,
akinhasan25@gmail.com 
}

\date{\today }%

\begin{abstract}
In this paper, we consider the Ising-Vanniminus model on an
arbitrary order Cayley tree. We generalize the results conjectured
in \cite{Akin2016,Akin2017} for an arbitrary order Cayley tree.
We establish existence and a full classification of translation
invariant Gibbs measures with memory of length 2 associated with
the model on arbitrary order Cayley tree. We construct the
recurrence equations corresponding generalized ANNNI model. We
satisfy the Kolmogorov \emph{consistency} condition. We propose a
rigorous measure-theoretical approach to investigate the Gibbs
measures with memory of length 2 for the model. We explain whether
the number of  branches of tree does not change the number of
Gibbs measures. Also we take up with trying to determine when
phase transition does occur.
\\

\end{abstract}
\maketitle
\section{Introduction}
One of the main purposes of equilibrium statistical mechanics
consists in describing all limit Gibbs distributions corresponding
to a given Hamiltonian \cite{Georgii}. One of the methods used for
the description of Gibbs measures on Cayley trees is Markov random
field theory and recurrent equations of this theory
\cite{AT1,NHSS,GRS,RAU,Rozikov,MDA,AGUT}. The approach we use here
is based on the theory of Markov random fields on trees and
recurrent equations of this theory. In this paper, we discuss
their relation with the recurrent equations of the theory of
Markov random fields on trees for Ising model \cite{NHSS,NHSS1}.
In \cite{ART}, we obtain a new set of limiting Gibbs measures for
the Ising model on a Cayley tree. In \cite{GU2011,UGAT,GATUN}, the
authors study the phase diagram for the Ising model on a Cayley
tree of arbitrary order $k$ with competing interactions.  In
\cite{GATUN}, the authors characterized each phase by a particular
attractor and the obtained the phase diagram  by following the
evolution and detecting the qualitative changements of these
attractors.

$n$-dimensional integer lattice, denoted $\mathbf{Z}^{n}$, has
so-called amenability property. Moreover, analytical solutions
does not exist on such lattice. But investigations of phase
transitions of spin models on hierarchical lattices showed that
there are exact calculations of various physical quantities (see
for example, \cite{Bax,Rozikov}). Such studies on the hierarchical
lattices begun with the development of the Migdal-Kadanoff
renormalization group method where the lattices emerged as
approximants of the ordinary crystal ones. On the other hand, the
study of exactly solved models deserves some general interest in
statistical mechanics \cite{Zachary}.

A Cayley tree is the simplest hierarchical lattice with
non-amenable graph structure \cite{Preston}. Also, Cayley trees
still play an important role as prototypes of graphs
\cite{Ostilli}. This means that the ratio of the number of
boundary sites to the number of interior sites of the Cayley tree
tends to a nonzero constant in the thermodynamic limit of a large
system. Nevertheless, the Cayley tree is not a realistic lattice,
however, its amazing topology makes the exact calculations of
various quantities possible.

One of the most interesting problems in statistical mechanics on a
lattice is the phase transition problem, i.e. deciding whether
there are many different Gibbs measures associated to a given
Hamiltonian \cite{BRZ,BG,Nasir,AGUT}. Investigations of phase
transitions of spin models on hierarchical lattices showed that
they make the exact calculation of various physical quantities
\cite{Sinai,GRRR,GHRR}. It was established the existence of the
phase transition, for the model in terms of finitely correlated
states, which describes ground states of the model. Up to this day
many authors have studied the existence of phase transition by
means of the recurrence equations corresponding to the
Ising-Vanniminus model on Cayley tree of order two and three
\cite{Akin2016,Akin2017,ART,NHSS,RAU}. Recently, Ganikhodjaev
\cite{Nasir} has studied the existence of phase transition for
Ising model on the semi-infinite Cayley tree of second order with
competing interactions up to third-nearest-neighbor generation
with spins belonging to the different branches of the tree. In the
present paper, for a given Hamiltonian, we provide a more general
construction of Gibb measures associated with the Hamiltonian. We
prove the existence of translation-invariant Gibb measures
associated to the model which yield the existence of the phase
transition.

It is well known that the Potts model is a generalization of the
Ising model, but the Potts model on a Cayley tree is not well
studied, compared to the Ising model \cite{Rozikov}. In the last
decade, many researches have investigated Gibbs measures
associated with Potts model on Cayley trees
\cite{GTA,AS2015,AS1,GATTMJ,GTACUBO,AT1}. In \cite{AS2015}, we
studied the existence, uniqueness and non-uniqueness of the Gibbs
measures associated with the Potts model on a Bethe lattice of
order three with three coupling constants by using Markov random
field method. In \cite{AS1}, we have obtained the exact solution
of a phase transition problem by means of Gibbs state of the same
Potts model in \cite{AS1}.

In the present paper, we are concerned with the Ising-Vanniminus
model on an arbitrary order Cayley tree. We investigate
translation invariant Gibbs measures associated with
Ising-Vannimenus model on arbitrary order Cayley tree. We
generalize the results obtained in \cite{Akin2016,Akin2017}. We
use the Markov random field method to describe the Gibbs measures.
We satisfy the Kolmogorov \emph{consistency} condition.  We
propose a rigorous measure-theoretical approach to investigate the
Gibbs measures with memory of length 2 corresponding to the
Ising-Vanniminus model on a Cayley tree of arbitrary order. Also
we take up with trying to determine when phase transition does
occur.

The outline of this paper is as follows. In Section
\ref{PRELIMINARIES} we give the definitions of the Cayley tree,
Gibbs measures and Ising-Vannimenus model. Section
\ref{Construction} provides a construction of Gibbs measures on an
arbitrary order Cayley tree. In Section \ref{even} we establish
the existence, uniqueness and non-uniqueness of the
translation-invariant Gibbs measures by means of the recurrence
equations for $k$-even, while in Section \ref{odd} we do the same
for $k$-odd. We contain in Section \ref{Conlussion} concluding
remarks and discussion of the consequences of the results with
next problems.
\section{PRELIMINARIES}\label{PRELIMINARIES}
\subsection{Cayley trees} Cayley trees (or Bethe lattices) are simple connected
undirected graphs $G = (V, E)$ ($V$ set of vertices, $E$ set of
edges) with no cycles (a cycle is a closed path of different
edges), i.e., they are trees \cite{Ostilli}. Let $\Gamma^k=(V, L,
i)$ be the uniform Cayley tree of order $k$ with a root vertex
$x^{(0)}\in V$, where each vertex has $k + 1$ neighbors with $V$
as the set of vertices and the set of edges. The notation $i$
represents the incidence function corresponding to each edge
$\ell\in L$, with end points $x_1,x_2\in V$. There is a distance
$d(x, y)$ on $V$ the length of the minimal point from $x$ to $y$,
with the assumed length of 1 for any edge (see Figure
\ref{cayley-level2}).

We denote the sphere of radius $n$ on $V$ by $ W_n=\{x\in V:
d(x,x^{(0)})=n \}$ and the ball of radius $n$ by $ V_n=\{x\in V:
d(x,x^{(0)})\leq n \}.$ The set of direct successors of any vertex
$x\in W_n$ is denoted by $S_k(x)=\{y\in W_{n+1}:d(x,y)=1 \}.$
\begin{figure} [!htbp]\label{cayley-level2}
\centering
\includegraphics[width=75mm]{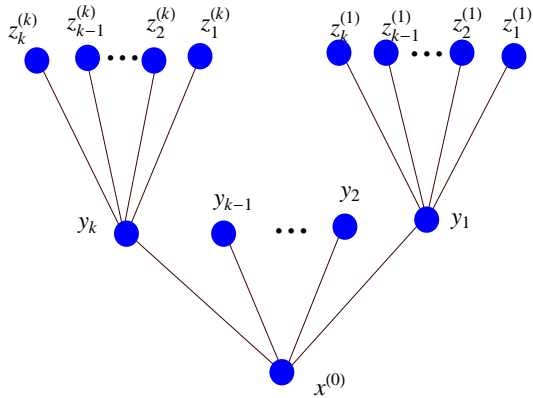}
\caption{Two successive generations of semi-infinite Cayley tree
$\Gamma^k$ of arbitrary order $k>1$ (branching ratio is finite
$k$). The fixed vertex $x^{(0)}$ is the root of the lattice that
emanates $k$ edges of $\Gamma^k$ ($y_j\in S(x^{(0)}),z_i^{(j)}\in
S(y_j)$).}\label{cayley-level2}
\end{figure}
\subsection{Ising-Vannimenus model}
The Ising  model with competing nearest-neighbors interactions is
defined by the Hamiltonian
\begin{equation}\label{hm1}
H(\sigma)=-J\sum_{<x,y>\subset V}\sigma(x)\sigma(y),
\end{equation}
where the sum runs over nearest-neighbor vertices $<x,y>$ and the
spins $\sigma(x)$ and $\sigma(y)$ take values in the set
$\Phi=\{-1,+1\}$.

The Hamiltonian
\begin{equation}\label{hm}
H(\sigma)=-J_p\sum_{>x,y<}\sigma(x)\sigma(y)-J\sum_{<x,y>}\sigma(x)\sigma(y)
\end{equation}
defines the \textbf{Ising-Vannimenus model} with competing
nearest-neighbors and next-nearest-neighbors, with the sum in the
first term representing the ranges of all nearest-neighbors, where
$J_p,J\in \mathbb{R}$ are coupling constants corresponding to
prolonged next-nearest-neighbor and nearest-neighbor potentials
\cite{Vannimenus}.

\subsection{Gibbs measures}
A finite-dimensional distribution of measure $\mu$  in the volume
$V_n$ has been defined by formula
\begin{equation}\label{mu1}
\mu_n(\sigma_n)=\frac{1}{Z_{n}}\exp[-\frac{1}{T}H_n(\sigma)+\sum_{x\in
W_{n}}\sigma(x)h_{x}]
\end{equation}
with the associated partition function defined as
\begin{equation*}\label{mu}
Z_n=\sum_{\sigma_n \in
\Phi^{V_n}}\exp[-\frac{1}{T}H_n(\sigma)+\sum_{x\in
W_{n}}\sigma(x)h_{x}],
\end{equation*}
where the spin configurations $\sigma_n$ belongs to $\Phi^{V_n}$
and $h=\{h_x\in \mathbb{R},x\in V\}$ is a collection of real
numbers that define boundary condition (see \cite{BG,GRRR,GHRR}).
Physically, Eq. \eqref{mu} represents the first step of the
Bethe-Peierls approach \cite{Bethe}.  Bleher \cite{Bleher1990}
proved that the disordered Gibbs distribution \eqref{mu1} in the
ferromagnetic Ising model associated to the Hamiltonian
\eqref{hm1} on the Cayley tree is extreme for $T \geq T^{SG}_C$,
where $T^{SG}_C$ is the critical temperature of the spin glass
model on the Cayley tree, and it is not extreme for $T< T^{SG}_C$.
Previously, researchers frequently used memory of length 1 over a
Cayley tree to study Gibbs measures \cite{BG,GRRR,GHRR}.

Let $S=\{1,2,...,s\}$ be a finite state space. On the infinite
product space ${{S}^{\mathbb{Z}}}$, one can define the product
$\sigma$-algebra, which is generated by cylinder sets
$_{m}[i_1,...,i_N]$  of length $N$ based on the block
$(i_1,...,i_N)$ at the place $m$. We denote by
${{\mathfrak{M}}}({{S}^{\mathbb{Z}}})$ the set of all measures on
${{S}^{\mathbb{Z}}}$. The set of all $\sigma$-invariant measures
in ${{S}^{\mathbb{Z}}}$
is denoted by ${{\mathfrak{M}}_{\sigma }}({{S}^{\mathbb{Z}}})$,
where $\sigma$ is the shift transformation.

\begin{prop}\cite[(8.1) Proposition]{Denker}\label{Kolmo-Cons1}
For $\mu \in {{\mathfrak{M}}_{\sigma }}({{S}^{\mathbb{Z}}})$ the
following properties are valid:
\begin{enumerate}
    \item $\sum\limits_{i\in S}{\mu {{(}_{0}}[i])=1}$;
    \item $\mu (_{n}[i_{0},...,i_{k}])\geq 0$ for any block
$({{i}_{0}},{{i}_{1}},...,{{i}_{k}})\in S^{k+1}$ and any $n\in
\mathbb{Z}$;
    \item $\mu (_{n}[{{i}_{0}},...,{{i}_{k}}])=\sum\limits_{{{i}_{k+1}}\in S}{\mu
    {{(}_{n}}[{{i}_{0}},...,{{i}_{k}},{{i}_{k+1}}])}$;
    \item  $\mu (_{n}[{{i}_{0}},...,{{i}_{k}}])=\sum\limits_{{{i}_{-1}}\in S}{\mu{{(}_{n}}[{{i}_{-1}},{{i}_{0}},...,{{i}_{k}}])}$.
\end{enumerate}
\end{prop}
The proof of the Proposition \ref{Kolmo-Cons1} can clearly be
checked for both the Bernoulli and the Markov measures on
$\sigma$-algebra \cite{Denker}. By a special case of Kolmogoroff's
consistency theorem (see \cite{Denker}), these properties are
sufficient to define a measure. It is well known that a Gibbs
measure is a generalization of a Markov measure to any graph,
therefore any Gibbs measure should satisfy the conditions in the
Proposition \ref{Kolmo-Cons1}. In the next sections, we will show
that the Gibbs measure associated to the Ising-Vannimenus model
satisfies the conditions in the Proposition \ref{Kolmo-Cons1}.

Let us consider increasing subsets of the set of states for one
dimensional lattices \cite{FV} as follows:
$$\mathfrak{G}_1\subset \mathfrak{G}_2\subset... \subset \mathfrak{G}_n\subset...,
$$
where $\mathfrak{G}_n$ is the set of states corresponding to
non-trivial correlations between $n$-successive lattice points;
$\mathfrak{G}_1$ is the set of mean field states; and
$\mathfrak{G}_2$ is the set of Bethe-Peierls states, the latter
extending to the so-called Bethe lattices. All these states
correspond in probability theory to so-called Markov chains with
memory of length $n$ (see \cite{D1,FV,Zachary}).

In \cite{FV},  by using  the idea of Bayesian extension, Fannes
and Verbeure defined states known as a finite-block measure or as
Markov chains with memory of length $n$ on the lattices. Recently,
the author has studied the Gibbs measures with memory of length 2
associated to the Ising-Vannimenus model on the Cayley tree of
order two and three \cite{Akin2016,Akin2017}. The construction is
based on the idea in the Proposition \ref{Kolmo-Cons1}. In the
present paper, we are going to establish the existence of Gibbs
measures associated with the Ising-Vannimenus model with memory of
length 2 on the Cayley tree of \emph{arbitrary order}.

\section{Construction of Gibbs measures on Cayley
tree}\label{Construction}

In this section, we will presents the general structure of Gibbs
measures with memory of length 2 associated with the Hamiltonian
\eqref{hm}  on an arbitrary order Cayley tree. On non-amenable
graphs, Gibbs measures depend on boundary conditions
\cite{Rozikov}. This paper considers this dependency for Cayley
trees, the simplest of graphs.

An arbitrary edge $<x^{(0)},x^{(1)}>=\ell \in L$ deleted from a
Cayley tree $\Gamma _{1}^{k}$ and $\Gamma _{0}^{k}$ splits into
two components: semi-infinite Cayley tree $\Gamma _{1}^{k}$ and
semi-infinite Cayley tree $\Gamma _{0}^{k}$. This paper considers
a semi-infinite Cayley tree $\Gamma _{0}^{k}$. For a finite subset
$V_n$ of the lattice, we define the finite-dimensional Gibbs
probability distributions on the configuration space
$\Omega^{V_n}=\{\sigma_n=\{\sigma(x)=\pm 1, x\in V_n \}\}$ at
inverse temperature $\beta=\frac{1}{kT}$ by formula.

Let $x\in W_{n}$ for some $n$ and
$S(x)=\{{{y}_{1}},{{y}_{2}},\cdots ,{{y}_{k}}\}$ are the direct
successors of $x$, where ${{y}_{1}},{{y}_{2}},\cdots ,{{y}_{k}}\in
{{W}_{n+1}}$.

Denote ${{B}_{1}}(x)=\left(
\begin{matrix}
   {{y}_{k}},\cdots ,{{y}_{2}},{{y}_{1}}  \\
   x  \\
\end{matrix} \right)$ a unite semi-ball with a center $x$.
We denote the set of all spin configurations on $V_n$ by
$\Phi^{V_n}$ and the set of all configurations on unite semi-ball
$B_1(x)$ by $\Phi^{B_1(x)}$. One can get that the set
$\Phi^{B_1(x)}$ consists of ${{2}^{k+1}}$ configurations:
\[
\Phi^{B_{1}(x)}=\left\{\left(\begin{matrix}
i_{k},\cdots ,i_{2},i_{1}  \\
   i  \\
\end{matrix} \right):i,i_{1},i_{2},\cdots ,i_{k}\in \Phi  \right\}.
\]
Let
\[
\sigma_{S}({{x}^{(0)}})=\left( \begin{matrix}
   \sigma ({{y}_{k}}),\cdots ,\sigma ({{y}_{2}}),\sigma ({{y}_{1}})  \\
   \sigma ({{x}^{(0)}})  \\
\end{matrix} \right)
\]
be a configuration on the set ${{x}^{(0)}}\cup S({{x}^{(0)}})$ and
\[
\sigma_{S}({{y}_{i}})=\left( \begin{matrix}
   \sigma (z_{k}^{(i)}),\cdots ,\sigma (z_{2}^{(i)}),\sigma (z_{1}^{(i)})  \\
   \sigma ({{y}_{i}})  \\
\end{matrix} \right)
\]
be a configuration on the set ${{y}_{i}}\cup S({{y}_{i}}),$ ${y}_{i}\in S({{x}^{(0)}})$.
Let $\Omega(S)$ be the set of all such configurations.


We wish to consider a probability measure $\mu
_{\mathbf{h}}^{(n)}$ that is formally given by
\begin{equation}\label{Gibbs1}
\mu _{\mathbf{h}}^{(n)}(\sigma )=\frac{1}{Z_{\mathbf{h}}^{(n)}}
\exp [-\beta {{H}_{n}}(\sigma )+\sum\limits_{x\in
{{W}_{n-1}}}{\sum\limits_{y\in S(x)}{\sigma }}
(x)\prod\limits_{y\in S(x)}{\sigma
(y)h_{B_{1}(y);\sigma_{S}(y)}}],
\end{equation}
where $\beta =\frac{1}{{{k}_{B}}T}$, ${{k}_{B}}$ is the Boltzmann
constant and $h_{B_{1}(y);\sigma_{S}(y)}$ is a real-valued
function of $y\in V$. $\sigma_n: x\in V_n\to \sigma_n(x)$ and
$Z_{\mathbf{h}}^{(n)}$ corresponds to the following partition
function:
\begin{equation}\label{partition1}
Z_{\mathbf{h}}^{(n)}=\sum\limits_{{{\sigma }_{n}}\in {{\Phi
}^{{{V}_{n}}}}}{\exp }[-\beta H({{\sigma }_{n}})+\sum\limits_{x\in
{{W}_{n-1}}}{\sum\limits_{y\in S(x)}{\sigma }}
(x)\prod\limits_{y\in S(x)}{\sigma
(y){{h}_{B_{1}(y);\sigma_{S}(y)}}}].
\end{equation}
In this paper, we suppose that vector valued function
$\textbf{h}:V\rightarrow \mathbb{R}^{2(k+1)}$ is defined by
\begin{equation}\label{consistency}
\textbf{h}:(x,y_k,y_{k-1},\ldots, y_2,y_1)\rightarrow
\textbf{h}_{B_1(x)}=(h_{B_1(x);\sigma_{S}({{x}})}:y_i\in S(x)),
\end{equation}
where $h_{B_1(x);\sigma_{S}({{x}})}\in \mathbb{R}$, $x\in W_{n-1}$
and $y_i\in S(x).$

We will consider a construction of an infinite volume distribution
with given finite-dimensional distributions. More exactly, we will
attempt to find a probability measure $\mu$ on $\Omega$  that is
compatible with given measures $\mu_{\textbf{h}}^{(n)}$,
\emph{i.e.},
\begin{equation}\label{CM}
\mu(\sigma\in\Omega:
\sigma|_{V_n}=\sigma_n)=\mu^{(n)}_{\textbf{h}}(\sigma_n), \ \ \
\textrm{for all} \ \ \sigma_n\in\Omega^{V_n}, \ n\in \mathbf{N}.
\end{equation}
We say that the probability distributions $\mu_{\textbf{h}}^{n}$
satisfy the Kolmogorov consistency condition if for any
configuration $\sigma_{n-1}\in\Omega^{V_{n-1}}$
\begin{equation}\label{comp}
\sum_{\omega\in\Omega^{W_n}}\mu^{(n)}_{\textbf{h}}(\sigma_{n-1}\vee\omega)=\mu^{(n-1)}_{\textbf{h}}(\sigma_{n-1}).
\end{equation}
This condition implies the existence of a unique measure
$\mu_{\textbf{h}}^{(n)}$ defined on $\Omega$ with a required
condition \eqref{CM}. Such a measure $\mu_{\textbf{h}}^{(n)}$ is
called a Gibbs measure with memory of length 2 associated to the
model \eqref{hm}.

We note that first two conditions of Proposition \ref{Kolmo-Cons1}
is trivial to check for the measure in \eqref{Gibbs1}. The
condition (3) of Proposition \ref{Kolmo-Cons1} is the same as the
condition in \eqref{comp}. Therefore, it should be proved that the
Gibbs measure \eqref{Gibbs1} satisfies the condition (3) in the
Proposition \ref{Kolmo-Cons1}.

\section{The recurrence equations for $k$-even}\label{even}
Let $k$ be the positive even integer, where $k$ is the order of
the Cayley tree. It is reasonable, though, to assume that the
different branches are equivalent, as is usually done for models
on trees.

Let
\[
\sigma _{S}^{+}({{x}^{(0)}})=\left(
\begin{matrix}
   \sigma ({{y}_{k}}),\cdots ,\sigma ({{y}_{2}}),\sigma ({{y}_{1}})  \\
   +  \\
\end{matrix} \right)
\]
be a configuration in $\Phi^{B_1(x)}$ (see Fig.
\ref{cayley-level2}). Let $m$ be the number of spins down, i.e.,
$\sigma ({{y}_{i}})=-1$ on the first level ${{W}_{1}}$, where
$0\leq m\leq k$. Then $(k-m)$ is the number of spins up, i.e.,
$\sigma ({{y}_{i}})=+1$ on the first level ${{W}_{1}}$. Let
\[\sigma _{S}^{+}({{y}_{i}})=\left( \begin{matrix}
   \sigma (z_{k}^{(i)}),\cdots ,\sigma (z_{2}^{(i)}),\sigma (z_{1}^{(i)})\\
   +  \\
\end{matrix} \right)\] be a configuration in $\Omega(S)$. Let $m$ be the number of spins down, i.e.,
$\sigma (z_{j}^{(i)})=-1$ on the second level ${{W}_{2}}$, where
$0\leq m \leq k.$
Let
\[\sigma _{S}^{-}({{x}^{(0)}})=\left(
\begin{matrix}
   \sigma ({{y}_{k}}),\cdots ,\sigma ({{y}_{2}}),\sigma ({{y}_{1}})\\
   -  \\
\end{matrix} \right)
\] be a configuration in $\Omega(S)$.
Let $m$ be the number of spins down, i.e., $\sigma (y_{i})=-1$ on
the first level ${{W}_{1}}$, where $0\leq m\leq k.$ Let
\[\sigma _{S}^{-}({{y}_{i}})=\left(
\begin{matrix}
   \sigma (z_{k}^{(i)}),\cdots ,\sigma (z_{2}^{(i)}),\sigma (z_{1}^{(i)})  \\
   -  \\
\end{matrix} \right)
\] be a configuration in $\Omega(S)$ (see Fig. \ref{cayley-level2}). Let $m$ be the number of spins down, i.e.,
$\sigma (z_{j}^{(i)})=-1$ on the second level ${{W}_{2}}$, where
$0 \leq m \leq k.$

For clarity, denote the configuration of the set
$\Phi^{B_1(x^{(0)})}$ by
\[S_{m}^{(k-m)}(\sigma ({{x}^{(0)}}))=\left( \begin{matrix}
   \overbrace{++\cdots +}^{k-m}\overbrace{--\cdots -}^{m}  \\
   \sigma (x^{(0)})  \\
\end{matrix} \right).\]
From  the consistency condition \eqref{comp}, we can use the
following equation:

\begin{thm}\cite{Akin2016}\label{theorem1}
The measures  $\mu_{\h}^{(n)}(\s)$, $n=1,2,...,$ in \eqref{mu}
satisfy the compatibility condition \eqref{comp} if and only if
for any $n\in \mathbf{N}$  the following equations hold:
\begin{eqnarray*}
\exp
(h_{B_{1}(x^{(0)});S_{0}^{k}(+)}+h_{B_{1}(x^{(0)});S_{0}^{k}(-)})&=&\frac{\left(
\sum\limits_{i=0}^{k}\left(
\begin{array}{c}
 k \\
 i
\end{array}
\right)(ab)^{k-2i}(-1)^{i}{{h}_{{{B}_{1}}(y_i);S_{i}^{k-i}(+)}}\right)^{k}}{{{\left(
\sum\limits_{i=0}^{k}{\left(
\begin{array}{c}
 k \\
 i
\end{array}
\right){{\left( \frac{a}{b}
\right)}^{k-2i}}}(-1)^{i}{{h}_{{{B}_{1}}(y_i);S_{i}^{k-i}(+)}}
\right)}^{k}}}\\
\exp
(h_{B_{1}(x^{(0)});S_{0}^{k}(+)}+h_{B_{1}(x^{(0)});S_{k}^{0}(-)})&=&\frac{\left(
\sum\limits_{i=0}^{k}\left(
\begin{array}{c}
 k \\
 i
\end{array}
\right)(ab)^{k-2i}(-1)^{i}{{h}_{{{B}_{1}}(y_i);S_{i}^{k-i}(+)}}\right)^{k}}{\left(
\sum\limits_{i=0}^{k}\left(
\begin{array}{c}
 k \\
 i
\end{array}
\right)(ab)^{2i-k}{{(-1)}^{i+1}}{{h}_{{{B}_{1}}(y_i);S_{i}^{k-i}(-)}}\right)^{k}}\\
\exp
(h_{B_{1}(x^{(0)});S_{k}^{0}(+)}+h_{B_{1}(x^{(0)});S_{k}^{0}(-)})&=&\frac{
\left( \sum\limits_{i=0}^{k}\left(%
\begin{array}{c}
  k \\
  i \\
\end{array}%
\right)\left(\frac{b}{a}\right)^{k-2i}{{(-1)}^{i+1}}{{h}_{{{B}_{1}}(y_i);S_{i}^{k-i}(-)}}
\right)^{k}}{\left( \sum\limits_{i=0}^{k}\left(
\begin{array}{c}
 k \\
 i
\end{array}
\right)(ab)^{2i-k}{{(-1)}^{i+1}}{{h}_{{{B}_{1}}(y_i);S_{i}^{k-i}(-)}}\right)^{k}}
\end{eqnarray*}
where $a=e^{\beta J}$ and $b=e^{\beta J_p}$.
\end{thm}
The Theorem \ref{theorem1} partially confirms the conjecture
formulated in \cite{Akin2016}. Also, the proof of the Theorem
\ref{theorem1} can be done as similar to \cite{Akin2016}.

By means of the last equalities, from \eqref{Gibbs1} and
\eqref{partition1} we can get that
\begin{eqnarray}\label{mu2}
&&\exp (\sigma(x^{(0)})\prod_{y\in S(x^{(0)})} \sigma
(y)h_{B_1(x^{(0)});S_m^{k-m}(\sigma (x^{(0)}))})\\\nonumber
&=&L_2\sum _{\eta \in \Phi ^{W_2}}( \exp (\beta J\sum _{y_i\in
S(x^{(0)})} \sigma (y_i)\sum _{z_j^{(i)}\in S(y_i)} \eta
(z_j^{(i)}))\\\nonumber &&\times \exp(\beta J_p\sigma
(x^{(0)})\sum _{z_j^{(i)}\in S^2(x^{(0)})} \eta (z_j^{(i)})+\sum
_{y_i\in S(x^{(0)})} \sigma (y_i)\prod _{z_j^{(i)}\in S(y_i)} \eta
(z_j^{(i)})h_{B_1(y_i);S_m^{k-m}(\sigma (y_i))}))\nonumber
\end{eqnarray}
where ${{L}_{2}}=\frac{{{Z}_{1}}}{{{Z}_{2}}}. $

Consider the configuration $S_{0}^{k}(\sigma
({{x}^{(0)}})=+)=\left(
\begin{matrix}
   +,\cdots ,+,+  \\
   +  \\
\end{matrix} \right)
$.
For the sake of simplicity, assume such that $\exp
[(-1)^{m}{{h}_{{{B}_{1}}(y_i);S_{m}^{k-m}(\sigma
(y)=+1)}}]=u_{1+m}^{{{(-1)}^{m}}}$, we have
\begin{eqnarray}\label{Eq-eden1a}
u_{1}^{'}=\exp (h_{B_{1}(x^{(0)});S_{0}^{k}(+)})=L_{2}\left(
\sum\limits_{i=0}^{k}\left(
\begin{array}{c}
 k \\
 i
\end{array}
\right)(ab)^{-2i+k}u_{1+i}^{{{(-1)}^{i}}} \right)^{k}.
\end{eqnarray}
Now let us consider the configuration $S_{k}^{0}(\sigma
({{x}^{(0)}})=+)=\left(
\begin{matrix}
   -,\cdots ,-,-  \\
   +  \\
\end{matrix} \right)$ and let
$$
\exp [{{(-1)}^{m+1}}{{h}_{{{B}_{1}}(y_i);S_{m}^{k-m}(\sigma
(y_i)=-1)}}]=u_{k+2+m}^{{{(-1)}^{m+1}}},
$$
then we have
\begin{eqnarray}\label{Eq-eden1b}
u_{k+1}^{'}=\exp (h_{B_{1}(x^{(0)});S_{k}^{0}(+)})=L_{2}
\left( \sum\limits_{i=0}^{k}\left(%
\begin{array}{c}
  k \\
  i \\
\end{array}%
\right)\left(\frac{b}{a}\right)^{-2i+k}u_{2+i+k}^{{(-1)}^{1+i}}
\right)^{k}.
\end{eqnarray}
Similarly, for the configuration $S_{0}^{k}(\sigma
({{x}^{(0)}})=-)=\left( \begin{matrix}
   +,\cdots ,+,+  \\
   -  \\
\end{matrix} \right)$, one can obtain
\begin{eqnarray}\label{Eq-eden1c}
{{\left( u_{k+2}^{'} \right)}^{-1}}=\exp
(-h_{B_{1}(x^{(0)});S_{0}^{k}(-)})={{L}_{2}}{{\left(
\sum\limits_{i=0}^{k}{\left(
\begin{array}{c}
 k \\
 i
\end{array}
\right){{\left( \frac{a}{b}
\right)}^{-2i+k}}}u_{1+i}^{{{(-1)}^{i}}} \right)}^{k}}.
\end{eqnarray}
Lastly, for the configuration
$S_{k}^{0}(\sigma
({{x}^{(0)}})=-)=\left(
\begin{matrix}
   -,\cdots ,-,-  \\
   -  \\
\end{matrix} \right)$
we have
\begin{eqnarray}\label{Eq-eden1d}
{{\left( u_{2(k+1)}^{'} \right)}^{-1}}=\exp
(-h_{B_{1}(x^{(0)});S_{k}^{0}(-)})=L_{2}\left(
\sum\limits_{i=0}^{k}\left(
\begin{array}{c}
 k \\
 i
\end{array}
\right)(ab)^{2i-k}u_{2+i+k}^{(-1)^{1+i}} \right)^{k}.
\end{eqnarray}
From \eqref{Eq-eden1a}-\eqref{Eq-eden1d} we immediately get that
\begin{eqnarray}\label{Eq-eden1e}
e^{(-1)^{m}h_{B_{1}(x^{(0)});S_{m}^{k-m}(+)}}&=&\left(e^{h_{B_{1}(x^{(0)});S_{0}^{k}(+)}}
\right)^{\frac{k-m}{k}}\left(e^{h_{B_{1}(x^{(0)});S_{k}^{0}(+)}}
\right)^{\frac{m}{k}}\\
e^{(-1)^{m+1}h_{B_{1}(x^{(0)});S_{m}^{k-m}(-)}}&=&\left(e^{
-h_{B_{1}(x^{(0)});S_{0}^{k}(-)}} \right)^{\frac{k-m}{k}}\left(
e^{-h_{B_{1}(x^{(0)});S_{k}^{0}(-)}}\right)^{\frac{m}{k}}.\label{Eq-eden1f}
\end{eqnarray}
Through the introduction of the new variables
$v_{i}=(u_{i})^{\frac{1}{k}}$ in the equations
\eqref{Eq-eden1a}-\eqref{Eq-eden1f}, we derive the following
recurrence system:
\begin{eqnarray}\label{recurrence1a}
v_{1}^{'}&=&\sqrt[k]{L_2}{{\left(\frac{{{(ab)}^{2}}{{v}_{1}}+{{v}_{k+1}}}{ab}\right)}^{k}},\\\label{recurrence1b}
v_{k+1}^{'}&=&\sqrt[k]{L_2}{{\left(\frac{{{a}^{2}}{{v}_{k+2}}+{{b}^{2}}{{v}_{2(k+1)}}}{ab{{v}_{k+2}}{{v}_{2(k+1)}}}\right)}^{k}},\\\label{recurrence1c}
(v_{k+2}^{'})^{-1}&=&\sqrt[k]{L_2}{{\left(\frac{{{a}^{2}}{{v}_{1}}+{{b}^{2}}{{v}_{k+1}}}{ab}\right)}^{k}},\\\label{recurrence1d}
(v_{2(k+1)}^{'})^{-1}&=&\sqrt[k]{L_2}{{\left(\frac{{{(ab)}^{2}}{{v}_{k+2}}+{{v}_{2(k+1)}}}{(ab){{v}_{k+2}}{{v}_{2(k+1)}}}\right)}^{k}}.
\end{eqnarray}

\subsection{Translation-invariant Gibbs measures: Even
case}\label{TIGM}
In this subsection, we are going to focus on the
existence of translation-invariant Gibbs measures (TIGMs) by
analyzing the equations \eqref{recurrence1a}-\eqref{recurrence1d}.
Note that vector-valued function
\begin{equation}\label{consistency1}
\textbf{h}(x) = \{h_{B_1(x);S_{m}^{k-m}(\sigma (x))}:m\in
\{0,1,2,...,k\},\sigma (x)\in \Phi \}
\end{equation}
is considered as translation-invariant if
$h_{B_1(x);S_{m}^{k-m}(\sigma (x))}= h_{B_1(y);S_{m}^{k-m}(\sigma
(y))}$ for all $y\in S(x)$ and $\sigma (x)=\sigma (y)$ (see for
details \cite{Akin2016,Rozikov}). A translation-invariant Gibbs
measure is defined as a measure, $\mu_{\textbf{h}}$, corresponding
to a translation-invariant function $\textbf{h}$ (see for details
\cite{NHSS,Rozikov}). Here we will assume that $v'_{i}=v_{i}$ for
all $i\in \{1,\ldots,2(k+1)\}$.
\begin{rem}\label{even-consistency-rem}
By using the equations \eqref{Eq-eden1e} and \eqref{Eq-eden1f}, it
can be shown that if the vector-valued function $\textbf{h}(x)$
given in \eqref{consistency1} has the following form:
\begin{equation*}
\textbf{h}(x)=(p,-\frac{(k-1)p+q}{k},\cdots,-\frac{p+(k-1)q}{k},
q,r,-\frac{(k-1)r+s}{k},\cdots,-\frac{r+(k-1)s}{k},s),
\end{equation*}
where $p,q,r,s\in \mathbb{R}$, then the \emph{consistency}
condition \eqref{comp} is satisfied.
\end{rem}
Now, we want to find Gibbs measures for  considered case. To do
so, we introduce some notations. Define the transformation
\begin{eqnarray}\label{cavity-even}
\textbf{F}=(F_1,F_{k+1},F_{k+2},F_{2(k+1)}): \mathbf{R}^4_+
\rightarrow \mathbf{R}^4_+
\end{eqnarray}
such that
\begin{eqnarray*}
v'_1&=&F_1(v_1,v_{k+1},v_{k+2},v_{2(k+1)}),\\
v'_{k+1}&=&F_{k+1}(v_1,v_{k+1},v_{k+2},v_{2(k+1)}),\\
v'_{k+2}&=&F_{k+2}(v_1,v_{k+1},v_{k+2},v_{2(k+1)}),\\
v'_{2(k+1)}&=&F_{2(k+1)}(v_1,v_{k+1},v_{k+2},v_{2(k+1)}).
\end{eqnarray*}
The fixed points of the cavity equation
$\textbf{v}=\textbf{F}(\textbf{v})$ given in the equation
\eqref{cavity-even} describe the translation-invariant Gibbs
measures associated to the model corresponding to the Hamiltonian
\eqref{hm}, where $\textbf{v}=(v_{1}, v_{k+1},
v_{k+2},v_{2(k+1)})$ and $k$ is positive even integer.

Description of the solutions of the system of equations
\eqref{recurrence1a}-\eqref{recurrence1d} is rather tricky. Assume
that $v_{1}^{'}=v_{k+2}^{'}$ and $v_{k+1}^{'}=v_{2(k+1)}^{'}$ that
is
\begin{eqnarray*}
\exp
({{(-1)}^{0}}{{h}_{{{B}_{1}}({{x}^{(0)}});S_{0}^{k}(+)}})&=&\exp
(-{{h}_{{{B}_{1}}({{x}^{(0)}});S_{0}^{k}(-)}})\\
\exp
({{(-1)}^{k}}{{h}_{{{B}_{1}}({{x}^{(0)}});S_{k}^{0}(+)}})&=&\exp
(-{{h}_{{{B}_{1}}({{x}^{(0)}});S_{k}^{0}(-)}}).
\end{eqnarray*}

Below we will consider the following case when the system of
equations \eqref{recurrence1a}-\eqref{recurrence1d} is solvable
for set

\begin{eqnarray}\label{invariantA}
A=\left\{(v_{1}^{'}, v_{k+1}^{'}, v_{k+2}^{'},v_{2(k+1)}^{'})\in
\mathbb{R}_+^4:
v_{1}^{'}=v_{k+2}^{'}=\frac{1}{v_{k+1}^{'}}=\frac{1}{v_{2(k+1)}^{'}}
\right\}.
\end{eqnarray}

Divide the equation \eqref{recurrence1a} by the equation
\eqref{recurrence1c}, then we have
\begin{eqnarray}\label{recurrence2a}
(v_{1}^{'})={{\left(
\frac{{{(ab)}^{2}}{{v}_{1}}+{{v}_{k+1}}}{{{a}^{2}}{{v}_{1}}+{{b}^{2}}{{v}_{k+1}}}
\right)}^{\frac{k}{2}}}.
\end{eqnarray}
Similarly, divide  the equation \eqref{recurrence1b} by the
equation \eqref{recurrence1d}, then we get
\begin{eqnarray}\label{recurrence2b}
v_{k+1}^{'}={{\left(
\frac{{{a}^{2}}{{v}_{k+2}}+{{b}^{2}}{{v}_{2(k+1)}}}{{{(ab)}^{2}}{{v}_{k+2}}+{{v}_{2(k+1)}}}
\right)}^{\frac{k}{2}}}={{\left(
\frac{{{a}^{2}}{{v}_{1}}+{{b}^{2}}{{v}_{(k+1)}}}{{{(ab)}^{2}}{{v}_{1}}+{{v}_{(k+1)}}}
\right)}^{\frac{k}{2}}}.
\end{eqnarray}
For brevity, denote $a^2=c$ and $b^2=d$. From \eqref{recurrence2a}
and \eqref{recurrence2b}, if we assume as
$x'=v_{1}^{'}=\frac{1}{v_{k+1}^{'}}$ ($x>0$), then we obtain the
following dynamical system $f: \mathbb{R}^{+}\rightarrow
\mathbb{R}^{+}$
\begin{eqnarray}\label{dynamiceven-1}
x'=\left(\frac{1+c d x^{2}}{d+c x^{2}}\right)^{\frac{k}{2}}=:f(x).
\end{eqnarray}


\begin{rem}\label{Remark 4.2}
For $(v_1,v_{k+1},v_{k+2},v_{2(k+1)})\in A$, that is
$$
h_{B_1(x);S_{0}^{k}(+)}=h_{B_1(x);S_{0}^{k}(-)}=-h_{B_1(x);S_{k}^{0}(+)}=-h_{B_1(x);S_{k}^{0}(-)}
$$
the equation \eqref{dynamiceven-1} is valid. Also, one can verify
that $\textbf{F}(A)\subset A.$ In other words, the set $A$ in
\eqref{invariantA} is an invariant set under the mapping
$\textbf{F}$ in \eqref{cavity-even}.
\end{rem}

Let us investigate the fixed points of the function given in
\eqref{dynamiceven-1}, i.e., $x = f(x)$. In fact, we should show
that the system \eqref{dynamiceven-1} has at least one solution
with respect to $x'$ in the domain $\mathbb{R}^{+}$.
It is obvious that $f$ is bounded and thus the curve $y = f(x)$
must intersect the line $y =m x.$ Therefore, this construction
provides one element of a new set of Gibbs measures with memory of
length 2 associated to the model \eqref{hm} for any $x\in
\mathbb{R}^{+}$ (see \cite[Proposition 10.7]{Preston}).

\begin{prop}\label{proposition-even}
The equation \eqref{dynamiceven-1} (with $x \geq 0, c> 0, d> 0$)
has one solution if  $d<1$. If $d
> \sqrt{\frac{k+1}{k-1}}$ then there exists $\eta_1(c,d)$, $\eta_2(c,d)$ with
$0<\eta_1(c,d)<\eta_2(c,d)$ such that equation
\eqref{dynamiceven-1} has 3 solutions if
$\eta_1(c,d)<m<\eta_2(c,d)$ and has 2 solutions if either
$\eta_1(c,d)=m$ or $\eta_2(c,d)=m$. In fact
\begin{eqnarray*}
\eta_i (c,d)=\frac{1}{x_i}\left(\frac{1+c d
x^2_i}{d+cx^2_i}\right)^{k/2}.
\end{eqnarray*}
where $x_1,x_2$ are the solutions of equation $c^2 d
x^4-c\left(d^2k-1-d^2-k \right) x^2+d=0.$
\end{prop}
\begin{proof} Let us consider the equation \eqref{dynamiceven-1}.
Taking the first and the second derivatives of the function $f$,
then we have
$$
f'(x)=\frac{c (d^2-1)k x \left(1+c d
x^2\right)^{-1+k/2}}{\left(d+c x^2\right)^{1+k/2} }
$$
and
$$
f''(x)=-\frac{c(d^2-1)k\left(3 c^2 d x^4+c\left(1+d^2+ k- d^2
k\right) x^2-d\right)}{\left(d+c x^2\right)^{2+\frac{k}{2}}
\left(1+c d x^2\right)^{-\frac{k}{2}+2}}.
$$
If $d< 1$, i.e. $J_p<0$, then $f$ is decreasing and the equation
\eqref{dynamiceven-1} has only a unique solution; thus we can
restrict ourselves to the case $d>1$.

Let us consider equation
\begin{equation}\label{Eq12}
3 c^2 d x^4+c\left(1+d^2+ k- d^2 k\right) x^2-d=0.
\end{equation}

It is clear that $3 c^2 d x^4+c\left(1+d^2+ k- d^2 k\right) x^2-d$
is even function. Solving such an equation w.r.t. $x$, we can find
a positive root
$$
x^*=\frac{\sqrt{-1-d^2-k+d^2 k+\sqrt{12 d^2+\left(1+d^2+k-d^2
k\right)^2}}}{\sqrt{6cd}}.
$$
$x^*$ is a unique positive root of the quartic equation
\eqref{Eq12}. Therefore, the function $f$ is convex up, if
$$x< \frac{\sqrt{-1-d^2-k+d^2 k+\sqrt{12 d^2+\left(1+d^2+k-d^2
k\right)^2}}}{\sqrt{6cd}}.
$$
The function $f$ is convex down, for
$$x>\frac{\sqrt{-1-d^2-k+d^2 k+\sqrt{12 d^2+\left(1+d^2+k-d^2
k\right)^2}}}{\sqrt{6cd}}.
$$
It is quite easy to see that three is more than one solution if
and only if there is more than one solution to $xf'(x)=f(x)$,
which is the same as
$$
c^2 d x^4-c\left(-1-d^2-k+d^2 k\right) x^2+d=0
$$
with the help of a little elementary analysis the proof is readily
completed.
\end{proof}
\begin{rem}
It is clear that the function \eqref{dynamiceven-1} has a unique
inflection point $x^{*}$ in the region $(0,\infty)$, therefore the
function \eqref{dynamiceven-1} has at most three fixed points in
the region $(0,\infty)$. We can conclude that the increase of $k$
affects the number of fixed points by no more than 3. So, we can
obtain at most 3 TIGMs associated to the model \eqref{hm} for
$(v_1,v_{k+1},v_{k+2},v_{2(k+1)})\in A$ in \eqref{invariantA}.
\end{rem}

\subsection{Numerical Example: Even Case}
Previously documented analysis can analytically solve these
equations for some given values $J,J_p, T$ and $k$, which we will
not show all of solutions here due to the complicated nature of
formulas and coefficients \cite{Wolfram}. In order to describe the
number of the fixed points of the function \eqref{dynamiceven-1},
we have manipulated the function \eqref{dynamiceven-1} and the
linear function $y=x$ via Mathematica \cite{Wolfram}.
We have obtained at most 3 positive real roots for some parameters
$J$ and $J_p$ (coupling constants), temperature $T$ and even
positive integer $k$.

Let us give an illustrative example.  Figs. \ref{fig1Cregion}
(a)-(b) show that there are 3 positive fixed points of the
function \eqref{dynamiceven-1}, if we take $J=-5.8,
J_p=3.25,T=14.358$ and $k=12,10$. Therefore, the phase transition
for the model \eqref{hm} occur.
\begin{figure} [!htbp]\label{fig1Cregion}
\centering
\includegraphics[width=65mm]{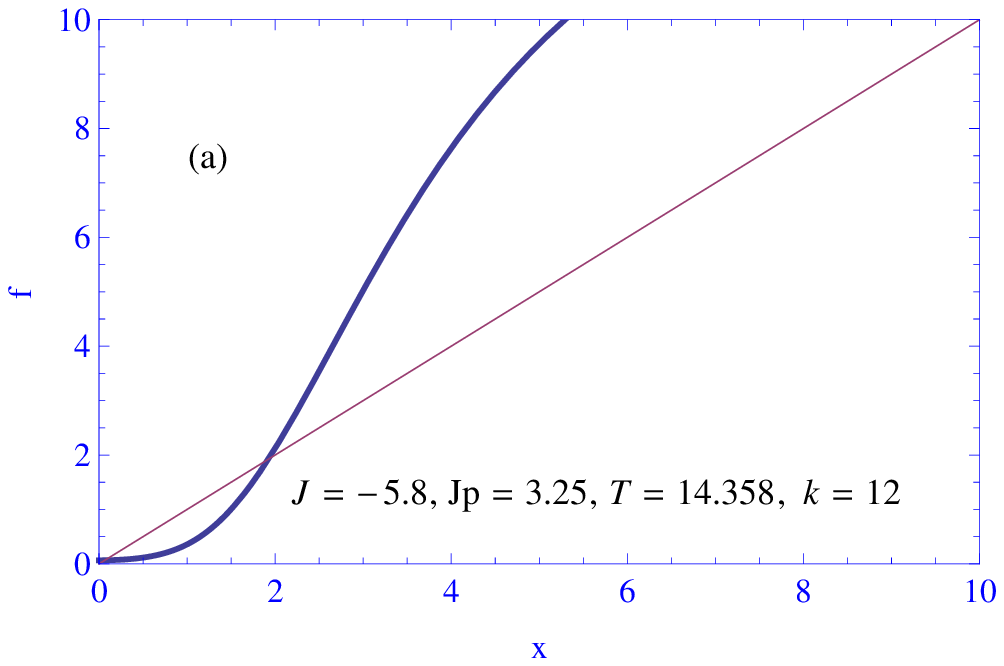}\ \ \label{fig1c}
\includegraphics[width=65mm]{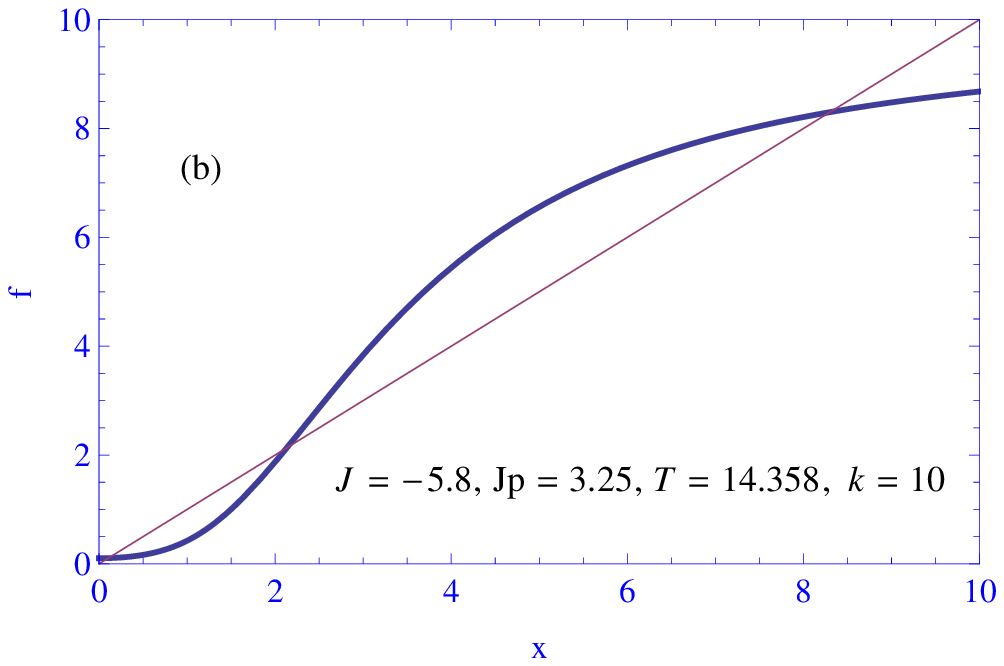}\ \
\includegraphics[width=65mm]{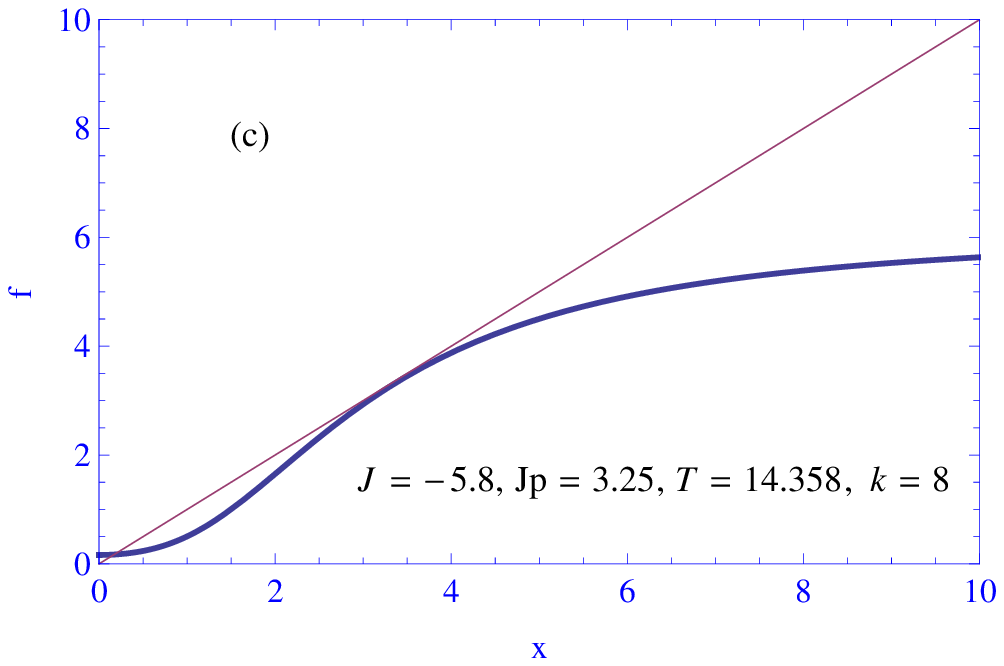}\ \
\includegraphics[width=65mm]{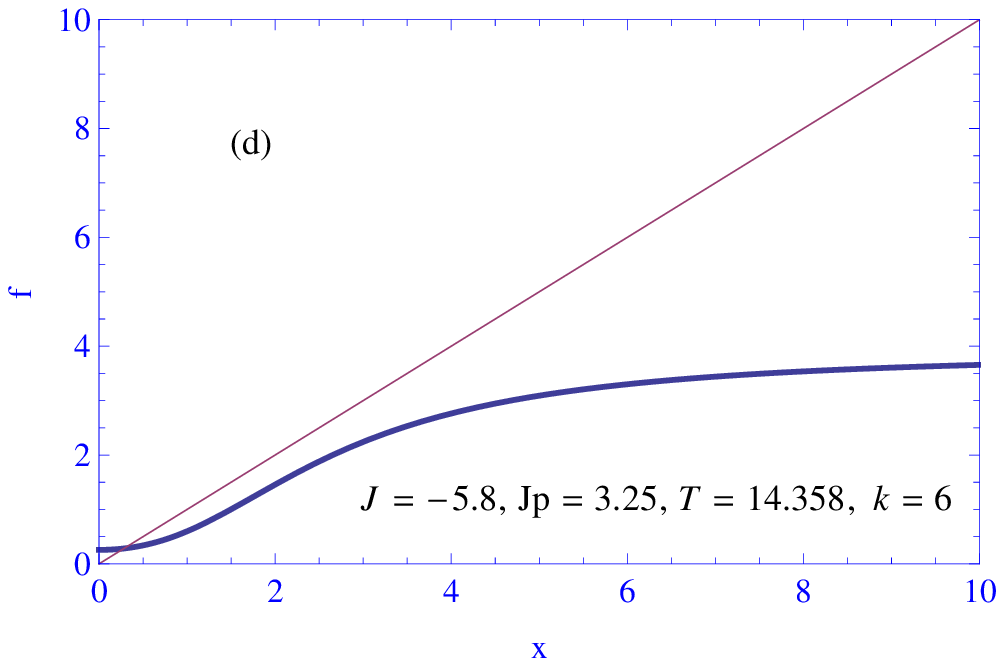}
\caption{Graphs of the function $f$ defined in
\eqref{dynamiceven-1} for $J=-5.8, J_p=3.25,T=14.358$ and even
integers $k=12,10,8,6$.}\label{fig1Cregion}
\end{figure}
In Figure \ref{fig1Cregion} (c), there exists two positive fixed
point of the function \eqref{dynamiceven-1} for $J=-5.8,
J_p=3.25,T=14.358$ and $k=8$.

In Figure \ref{fig1Cregion} (d), we are also able to find that,
for the parameters $J=-5.8, J_p=3.25,T=14.358$ and $k=6$, the
function \eqref{dynamiceven-1} has a unique positive fixed point.
Therefore,  the phase transition does not occur for $J=-5.8,
J_p=3.25,T=14.358$ and $k=6$.

We note that for $J=-5.8, J_p=3.25,T=14.358$ and $k=10$, the
function \eqref{dynamiceven-1} have three positive fixed points as
$x_1^{*}=0.106457, x_2^{*}=2.13383, x_3^{*}=8.30085.$ Figure
\ref{fig1Cregion} (b) shows that for all $x\in (x_2^{*},x_3^{*})$,
$\lim\limits_{n\rightarrow \infty}f^n(x)=x_3^{*}.$ Similarly, for
all $x\in (x_1^{*},x_2^{*})$, $\lim\limits_{n\rightarrow
\infty}f^n(x)=x_1^{*}.$ Therefore, the fixed points $x_1^{*}$ and
$x_3^{*}$ are stable and $x_2^{*}$ is unstable.

Therefore, there is a critical temperature $T_c > 0$ such that for
$T < T_c$ this system of equations has 3 positive solutions:
$h_1^{*}; h_2^{*};h_3^{*}.$  We denote the Gibbs measure that
corresponds to the root $h_1^{*}$ (and respectively
$h_2^{*};h_3^{*}$) by $\mu^{(1)}$ (and respectively
$\mu^{(2)}$,$\mu^{(3)}$).

\begin{rem}\label{exrema1}
We can conclude that the Gibbs measures  $\mu_1^{*}$ and
$\mu_3^{*}$ corresponding to the stable fixed points $x_1^{*}$ and
$x_3^{*}$ are extreme Gibbs distributions (for details
\cite{Iof,NHSS}).
\end{rem}
\begin{rem}
For $J=-5.8, J_p=3.25,T=14.358$ and $k>6$ ($k$ is even
integer), the model has phase transition. For $J=-5.8,
J_p=3.25,T=14.358$ and $k=6$, the phase transition of the model
does not occur.
\end{rem}

\section{The Recurrence equations for $k$-odd.}\label{odd}
Let us derive the recurrence equations to describe the existence
of the translation-invariant Gibbs measures (TIGMs) associated to
the model \eqref{hm} on the Cayley tree of order $k$-odd.

From the equations \eqref{Gibbs1}, \eqref{partition1} and
\eqref{mu2}, one get the following equations:
\begin{eqnarray}\label{odd1a}
u_{1}^{'}&=&\exp (h_{B_{1}(x^{(0)});S_{0}^{k}(+)})=L_{2}\left(\sum
_{i=0}^k (a b)^{-2 i+k} \left(
\begin{array}{c}
 k \\
 i
\end{array}
\right)u_{1+i}^{(-1)^i}\right)^{k}\\
(u_{k+1}^{'})^{-1}&=&\exp
(-h_{B_{1}(x^{(0)});S_{k}^{0}(+)})={{L}_{2}}{ {\left(
\sum\limits_{i=0}^{k}{\left(
\begin{array}{c}
 k \\
 i
\end{array}
\right){{\left( \frac{b}{a}
\right)}^{-2i+k}}}u_{2+i+k}^{{{(-1)}^{1+i}}}
\right)}^{k}}.\label{odd1b}
\end{eqnarray}
For the configuration $S_{0}^{k}(\sigma ({{x}^{(0)}})=-)=\left(
\begin{matrix}
   \overbrace{+,\cdots ,+,+}^{k-odd}  \\
   -  \\
\end{matrix} \right)$, similarly to \eqref{odd1a} and \eqref{odd1b} we obtain
\begin{eqnarray}\label{odd1c}
{{\left( u_{k+2}^{'} \right)}^{-1}}=\exp
(-{{h}_{{{B}_{1}}({{x}^{(0)}});S_{0}^{k}(-)}})={{L}_{2}}{{\left(
\sum\limits_{i=0}^{k}{\left(
\begin{array}{c}
 k \\
 i
\end{array}
\right){{\left( \frac{a}{b}
\right)}^{-2i+k}}}u_{1+i}^{{{(-1)}^{i}}} \right)}^{k}}.
\end{eqnarray}
Lastly, for the configuration $S_{k}^{0}(\sigma
({{x}^{(0)}})=-)=\left(
\begin{matrix}
   \overbrace{-,\cdots ,-,- }^{k-odd} \\
   -  \\
\end{matrix} \right)$  we have
\begin{eqnarray}\label{odd1d}
\left( u_{2(k+1)}^{'} \right)=\exp
({{h}_{{{B}_{1}}({{x}^{(0)}});S_{k}^{0}(-)}})={{L}_{2}}{{\left(
\sum\limits_{i=0}^{k}{\left(
\begin{array}{c}
 k \\
 i
\end{array}
\right){{(a b)}^{2i-k}}}u_{2+i+k}^{{{(-1)}^{1+i}}} \right)}^{k}}.
\end{eqnarray}
From the equations \eqref{odd1a}-\eqref{odd1d},  it is obvious
that
\begin{eqnarray}\label{even-eq4}
e^{(-1)^{m}h_{B_{1}(x^{(0)});S_{m}^{k-m}(+)}}&=&\left(e^{h_{B_{1}(x^{(0)});S_{0}^{k}(+)}}
\right)^{\frac{k-m}{k}}\left(e^{-h_{B_{1}(x^{(0)});S_{k}^{0}(+)}}
\right)^{\frac{m}{k}}\\\label{even-eq5}
e^{(-1)^{m+1}h_{B_{1}(x^{(0)});S_{m}^{k-m}(-)}}&=&\left(e^{
-h_{B_{1}(x^{(0)});S_{0}^{k}(-)}} \right)^{\frac{k-m}{k}}\left(
e^{h_{B_{1}(x^{(0)});S_{k}^{0}(-)}}\right)^{\frac{m}{k}}.
\end{eqnarray}
By substituting  variables $u_i=v_i^{k}$ for $i=1,2,\cdots,2(k+1)$
in the recurrent equations \eqref{odd1a}-\eqref{odd1d}, after
small calculations, we can express a new recurrence system in a
simpler form:
\begin{eqnarray}\label{odd4}
(v'_{1})&=&\sqrt[k]{L_2}\left(\frac{1+(a
b)^{2}v_{1}v_{k+1}}{abv_{k+1}}\right)^k,\\\label{odd4a}
(v'_{k+1})^{-1}&=&\sqrt[k]{L_2}\left(\frac{b^{2}+a^{2}v_{k+2}v_{2(k+1)}}{abv_{k+2}}\right)^k,\\\label{odd4b}
(v'_{k+2})^{-1}&=&\sqrt[k]{L_2}\left(\frac{b^{2}+a^{2}v_{1}v_{k+1}}{abv_{k+1}}\right)^k,\\\label{odd4c}
(v'_{2(k+1)})&=&\sqrt[k]{L_2}\left(\frac{1+(a
b)^{2}v_{k+2}v_{2(k+1)}}{abv_{k+2}}\right)^k.\label{odd4d}
\end{eqnarray}

\subsection{The translation-invariant Gibbs measures: Odd
case}\label{TIGM-odd} In this subsection, we will identify the
solutions of the system of nonlinear equations
\eqref{odd4}-\eqref{odd4c} to describe the translation invariant
Gibbs measures associated to the model \eqref{hm} on the arbitrary
odd-order Cayley tree.

\begin{rem}\label{odd-consistency-rem}
By using the equations \eqref{even-eq4} and \eqref{even-eq5}, it
can be shown that if the vector-valued function $\textbf{h}(x)$
given in \eqref{consistency1} has the following form:
\begin{equation*}
\textbf{h}(x)=(p,\cdots,\frac{(-1)^m((k-m)p-mq)}{k},\cdots,
q,r,\cdots,\frac{(-1)^{m+1}(ms-(k-m)r)}{k},\cdots,s),
\end{equation*}
where $p,q,r,s\in \mathbb{R}$, then the \emph{consistency}
condition \eqref{comp} is satisfied.
\end{rem}
Now, we want to find the TIGMs for considered case. To do so, we
introduce some notations. Define transformation
\begin{eqnarray}\label{cavity}
\textbf{F}=(F_1, F_{k+1}, F_{k+2}, F_{2(k+1)}): \mathbf{R}^4_+
\rightarrow \mathbf{R}^4_+
\end{eqnarray}
such that
\begin{eqnarray*}
v'_1&=&F_1(v_1,v_{k+1},v_{k+2},v_{2(k+1)}),\\
v'_{k+1}&=&F_{k+1}(v_1,v_{k+1},v_{k+2},v_{2(k+1)}),\\
v'_{k+2}&=&F_{k+2}(v_1,v_{k+1},v_{k+2},v_{2(k+1)}),\\
v'_{2(k+1)}&=&F_{2(k+1)}(v_1,v_{k+1},v_{k+2},v_{2(k+1)}).
\end{eqnarray*}
The fixed points of the cavity equation
$\textbf{v}=\textbf{F}(\textbf{v})$ given in the Eq.
\eqref{cavity} describe the translation invariant Gibbs measures
associated to the model \eqref{hm}, where
$\textbf{v}=(v_1,v_{k+1},v_{k+2},v_{2(k+1)})$ and $k$ is any
positive odd integer greater than 1.

Divide \eqref{odd4} by \eqref{odd4a}, then we have
\begin{eqnarray}\label{odd4e}
v_{k+1}^{2k+1}v_{k+2}^{-k}=\left(\frac{1+(ab)^2v_{k+1}^{k+1}}{b^2+a^2v_{k+2}^{k+1}}\right)^k.
\end{eqnarray}
Similarly, divide \eqref{odd4c} by \eqref{odd4b}, then one gets
\begin{eqnarray}\label{odd4f}
v_{k+2}^{2k+1}v_{k+1}^{-k}=\left(\frac{1+(ab)^2v_{k+2}^{k+1}}{b^2+a^2v_{k+1}^{k+1}}\right)^k.
\end{eqnarray}
Multiply the equations \eqref{odd4e} and \eqref{odd4f}, we obtain
\begin{eqnarray*}\label{recurrence5}
v_{k+1}^{k+1}v_{k+2}^{k+1}=\left(\frac{1+(ab)^2v_{k+1}^{k+1}}
{b^2+a^2v_{k+2}^{k+1}}\right)^k\left(\frac{1+(ab)^2v_{k+2}^{k+1}}{b^2+a^2v_{k+1}^{k+1}}\right)^k.
\end{eqnarray*}
Let us consider set
\begin{eqnarray}\label{invariantB}
B=\left\{(v_1,v_{k+1},v_{k+2},v_{2(k+1)})\in
\mathbb{R}_+^4:v_{1}=v_{k+1}^{k}=v_{2({k+1})}=v_{{k+2}}^{k}\right\}.
\end{eqnarray}
Assume that $v_{k+1}^{k+1}=v_{k+2}^{k+1}=x$, and $a^2=c$, $b^2=d$
then we get
\begin{eqnarray}\label{recurrence5a}
x=\left(\frac{1+c d x}{d+c x}\right)^k=:g(x).
\end{eqnarray}
\begin{rem} If $(v_1,v_{k+1},v_{k+2},v_{2(k+1)})\in B$,
that is
$$
h_{B_1(x);S_{0}^{k}(+)}=kh_{B_1(x);S_{k}^{0}(+)}=kh_{B_1(x);S_{0}^{k}(-)}=h_{B_1(x);S_{k}^{0}(-)}
$$
then the equation \eqref{recurrence5a} is valid. Also, one can
verify that $\textbf{F}(B)\subset B.$ That is, the set $B$ is an
invariant set under the mapping $\textbf{F}$.
\end{rem}

Now we examine how many solutions the equation $g(x) = x$ has.
Thus, similarly to the Proposition \ref{proposition-even}, we have
the following Proposition. Here, by using the procedure given in
\cite[Proposition 10.7]{Preston} we will describe the number of
fixed points of the function $g$ in \eqref{recurrence5a}.
\begin{prop}\label{proposition-odd}
The equation \eqref{recurrence5a} (with $x \geq 0, c > 0, d > 0$)
has one solution if  $d<1$. If $d
> \frac{k+1}{k-1}$ then there exists $\eta_1(c,d)$, $\eta_2(c,d)$ with
$0<\eta_1(c,d)<\eta_2(c,d)$ such that equation
\eqref{dynamiceven-1} has 3 solutions if
$\eta_1(c,d)<m<\eta_2(c,d)$ and has 2 solutions if either
$\eta_1(c,d)=m$ or $\eta_2(c,d)=m$. In fact
\begin{eqnarray*}
\eta_i (c,d)=\frac{1}{x_i}\left(\frac{1+c d
x_i}{d+cx_i}\right)^{k}.
\end{eqnarray*}
where $x_1,x_2$ are the solutions of quadratic equation $c^2 d
x^2-c(d^2(k-1)-(1+ k))x+d=0.$
\end{prop}
The proof of Proposition \ref{proposition-odd} can be done easily
by following the procedure in Proposition \ref{proposition-even}.
\begin{figure} [!htbp]\label{fig2Cregion}
\centering
\includegraphics[width=65mm]{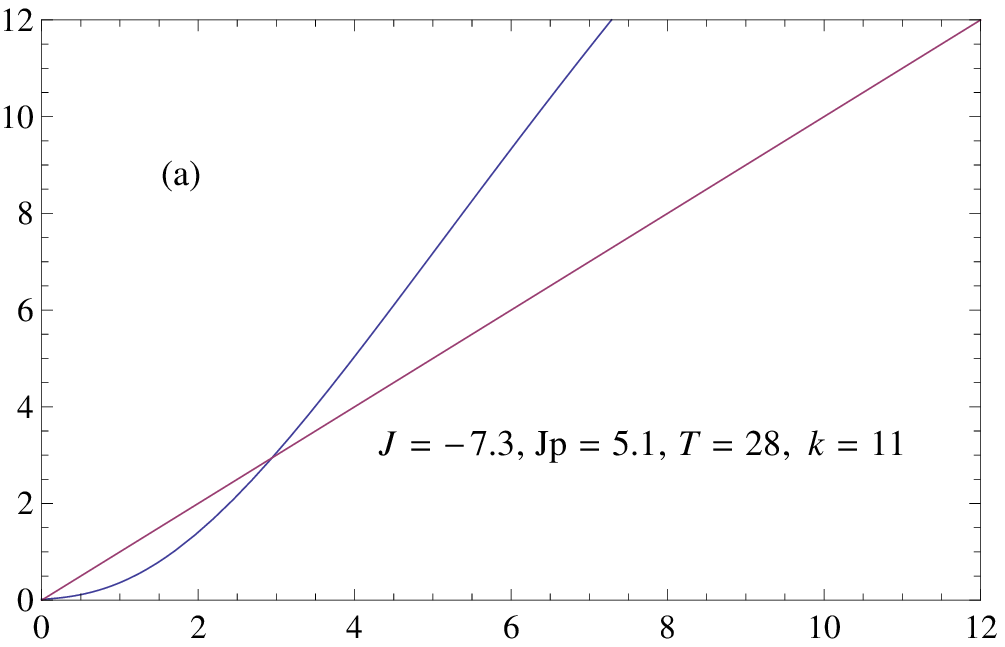}\ \ \label{fig1c}
\includegraphics[width=65mm]{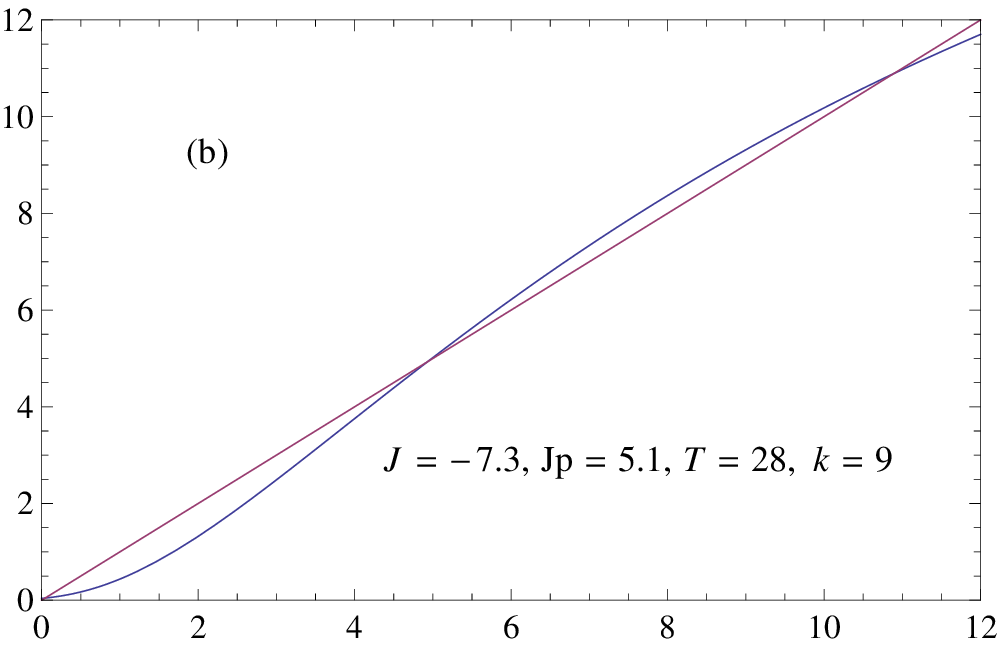}\ \
\includegraphics[width=65mm]{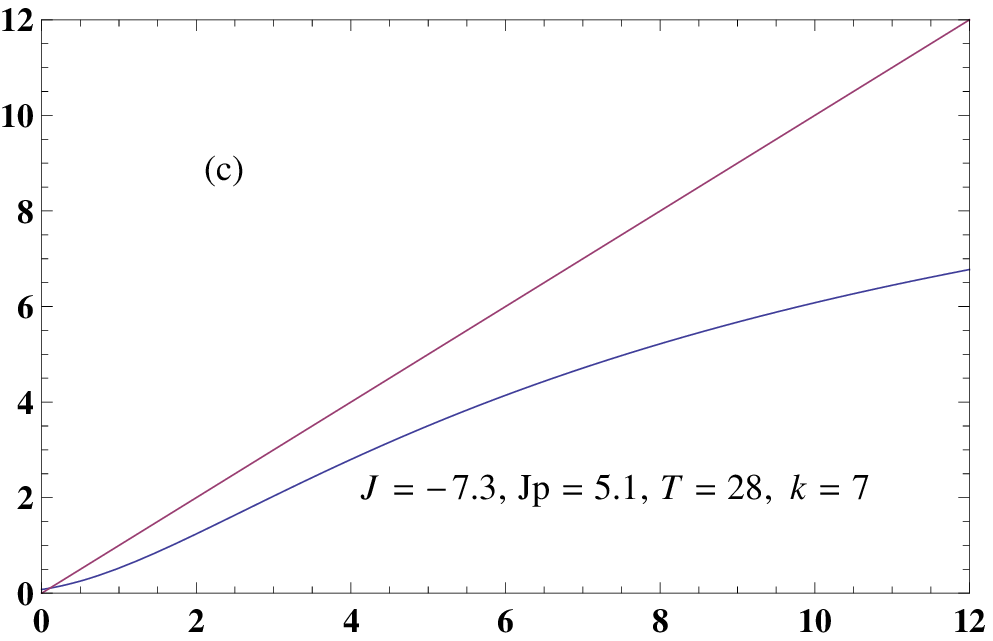}\ \
\caption{Graphs of the function $g$ given in \eqref{dynamiceven-1}
for $J = -7.3, Jp = 5.1, T = 28$ and odd integers $k=11,9,7$,
respectively.}\label{fig2Cregion}
\end{figure}
\subsection{Illustrative Example: Odd Case}
 We have manipulated the equation \eqref{recurrence5a} via Mathematica \cite{Wolfram}.
We have obtained at most 3 positive real roots for some parameters
$J$ and $J_p$ and temperature $T$. As an illustrative example, the
Figures \ref{fig2Cregion} (a)-(b) show that there are 3 positive
fixed points of the function \eqref{recurrence5a} for  $J = -7.3,
J_p = 5.1, T = 28$ and $k=11,9$ values. Therefore, we have
demonstrated the occurrence of phase transitions. The Figures
\ref{fig2Cregion} (a)-(b) shows that there are three positive
fixed points of the function $g$ for $J = -7.3, J_p = 5.1, T = 28$
and $k=11,9$.

In the Figure \ref{fig2Cregion} (c), there exists a unique
positive fixed point of the function \eqref{recurrence5a} for $J =
-7.3, J_p = 5.1, T = 28$ and $k=7$. Therefore,  the phase
transition does not occur for $J = -7.3, J_p = 5.1, T = 28$ and
$k=7$.

We can explicitly compute the fixed points of  the function
\eqref{recurrence5a} for given some parameters $J, J_p, T$ and
$k$. For example, for $J = -7.3, J_p = 5.1, T = 28$ and $k=9$, the
function \eqref{recurrence5a} have three positive fixed points as
$x_1^{*}=0.0448184, x_2^{*}=4.93008, x_3^{*}=10.8931.$ The Figure
\ref{fig2Cregion} (b) shows that for all $x\in (x_2^{*},x_3^{*})$,
$\lim\limits_{n\rightarrow \infty}g^n(x)=x_3^{*}.$ Similarly, for
all $x\in (x_1^{*},x_2^{*})$, $\lim\limits_{n\rightarrow
\infty}g^n(x)=x_1^{*}.$ Therefore, the fixed points $x_1^{*}$ and
$x_3^{*}$ are stable and $x_2^{*}$ is unstable.

\begin{rem}
As concluded in Remark \ref{exrema1}, we can see that the Gibbs
measures $\mu_1^{*}$ and $\mu_3^{*}$ corresponding to the stable
fixed points $x_1^{*}$ and $x_3^{*}$ are extreme Gibbs
distributions (for details \cite{Iof,NHSS}).
\end{rem}

\begin{rem} There is a critical temperature $T_c > 0$ such that for $T <
T_c$ the system of nonlinear equations \eqref{odd4}-\eqref{odd4c}
has 3 positive solutions: $h_1^{*}; h_2^{*};h_3^{*}.$  We denote
the Gibbs measure that corresponds to the root $h_1^{*}$ (and
respectively $h_2^{*};h_3^{*}$) by $\mu^{(1)}$ (and respectively
$\mu^{(2)}$,$\mu^{(3)}$).
\end{rem}

\section{Conclusions}\label{Conlussion}
In the present paper, we have proposed a rigorous
measure-theoretical approach to investigate the Gibbs measures
with memory of length 2 associated with the Ising-Vanniminus model
on the arbitrary order Cayley tree. We have generalized the
results conjectured in \cite{Akin2016,Akin2017} for an arbitrary
order Cayley tree. We have used the Markov random field method to
describe the Gibbs measures. We constructed the recurrence
equations corresponding generalized ANNNI model. We have satisfied
the Kolmogorov \emph{consistency} condition. We have explained
whether the number of  branches of tree does not change the number
of Gibbs measures. We have concluded that the order $k$ of the
tree significantly affects the occurrence of phase transition.
Also, we have seen that the role of $k$ is rather significant on
the number of Gibbs measures. Exact description of the solutions
of the system of recurrence equations
\eqref{recurrence1a}-\eqref{recurrence1d} (and
\eqref{odd4}-\eqref{odd4c}) is rather tricky. Therefore, we were
able to resolve only case \eqref{invariantA} (and
\eqref{invariantB}) for even $k$ (and odd $k$, respectively), the
other cases remain open problem. Also, depending on the even and
odd of $k$, the recurrence equations obtained for even branch
totaly differ from odd branch.

Note that for many problems the solution on a tree is much simpler
than on a regular lattice such as $d$-dimensional integer lattice
and is equivalent to the standard Bethe-Peierls theory
\cite{Katsura}. Although the Cayley tree is not a realistic
lattice; however, its amazing topology makes the exact
calculations of various quantities possible. Therefore, the
results obtained in our paper can inspire to study the Ising and
Potts models over multi-dimensional lattices or the grid
${{\mathbb{Z}}^{d}}$. After a glimpse of some applications, we
believe now that new theoretical developments can be inspired by
concrete problems. By considering the method used in this paper,
the investigation of Gibbs measures with memory of length $n>2$ on
arbitrary order Cayley tree and Cayley tree-like lattices
\cite{UA1,UA2,Moraal} is planned to be the subject of forthcoming
publications.


\end{document}